\documentclass[leqno]{article}

\usepackage[english]{babel}
\usepackage[utf8]{inputenc}
\usepackage{amsmath}
\usepackage{amssymb}
\usepackage{amsfonts}
\usepackage{enumerate}
\usepackage{vmargin}
\usepackage[all]{xy}
\usepackage{mathrsfs}
\usepackage{mathtools}
\usepackage{lmodern}
\usepackage{slashed}
\usepackage[colorlinks=true,linkcolor=blue,pagebackref=true]{hyperref}%
\setmarginsrb{3cm}{3cm}{3.5cm}{3cm}{0cm}{0cm}{1.5cm}{3cm}
\usepackage{comment}

\usepackage{accents}

\newlength{\dhatheight}

\newcommand{\Cbb}{\ensuremath{\mathbb{C}}}
\newcommand{\C}{\mathrm{C}}

\newcommand{\B}{\mathrm{B}} 

\newcommand{\Q}{\mathrm{Q}}

\let\L\relax 
\newcommand{\L}{\mathrm{L}}

\newcommand{\M}{\mathrm{M}}

\newcommand{\VN}{\mathrm{VN}} 
\newcommand{\la}{\langle}
\newcommand{\ra}{\rangle}
\newcommand{\WH}{\mathrm{WH}}
\newcommand{\WA}{\mathrm{WA}} 

\renewcommand{\leq}{\ensuremath{\leqslant}}
\renewcommand{\geq}{\ensuremath{\geqslant}}
\newcommand{\qed}{\hfill \vrule height6pt  width6pt depth0pt}

\newcommand{\norm}[1]{\left\Vert#1\right\Vert}

\newcommand{\co}{\colon}

\newcommand{\ot}{\otimes}
\newcommand{\ovl}{\overline}
\newcommand{\otvn}{\ovl\ot}

\newcommand{\ov}{\overset}

\newcommand{\epsi}{\varepsilon}
\renewcommand{\d}{\mathop{}\mathopen{}\mathrm{d}} 
\renewcommand{\d}{\mathop{}\mathopen{}\mathrm{d}}

\DeclareMathOperator{\supp}{supp} 

\newtheorem{thm}{Theorem}[section]

\newtheorem{prop}[thm]{Proposition}

\newtheorem{cor}[thm]{Corollary}
\newtheorem{lemma}[thm]{Lemma}

\newtheorem{remark}[thm]{Remark}
\newtheorem{example}[thm]{Example}

\newenvironment{proof}[1][]{\noindent {\it Proof #1} : }{\hbox{~}\qed
\smallskip
}

\usepackage{tocloft}
\numberwithin{equation}{section}
\usepackage[nottoc,notlot,notlof]{tocbibind}

\let\OLDthebibliography\thebibliography
\renewcommand\thebibliography[1]{
  \OLDthebibliography{#1}
  \setlength{\parskip}{0pt}
  \setlength{\itemsep}{0pt plus 0.3ex}
}

\newcommand\reallywidehat[1]{\arraycolsep=0pt\relax%
\begin{array}{c}
\stretchto{
  \scaleto{
    \scalerel*[\widthof{\ensuremath{#1}}]{\kern-.5pt\bigwedge\kern-.5pt}
    {\rule[-\textheight/2]{1ex}{\textheight}} 
  }{\textheight} %
}{0.5ex}\\           
#1\\                 
\rule{-1ex}{0ex}
\end{array}
}

\begin{document}
\selectlanguage{english}
\title{\bfseries{Multiplicativity of the weak Haagerup constant}}
\date{}
\author{\bfseries{C\'edric Arhancet}}
\maketitle

\begin{abstract}
We prove that the weak Haagerup constant is multiplicative with respect to direct products, namely $\Lambda_{\mathrm{WH}}(G \times H)=\Lambda_{\mathrm{WH}}(G) \Lambda_{\mathrm{WH}}(H)$ for any locally compact groups $G$ and $H$. Knudby previously proved the submultiplicative inequality and asked whether equality always holds. 
\end{abstract}

\makeatletter
\renewcommand{\@makefntext}[1]{#1}
\makeatother
\footnotetext{2020 {\it Mathematics subject classification:} 43A22, 22D05, 46L07.\\
{\it Key words}: weak Haagerup property, Herz--Schur multipliers, weak amenability, direct products.}

{
\hypersetup{linkcolor=blue}
\tableofcontents
}

\section{Introduction}
\label{sec:Introduction}

The weak Haagerup property was introduced by Knudby in \cite[Definition 1.1]{Knu14} and \cite[Definition 4.1 p.~3476]{Knu16} as an approximation property for locally compact groups which is weaker than both weak amenability and the Haagerup property. Recall that a locally compact group $G$ has the weak Haagerup property if there exist a constant $C > 0$ and a net $(u_\alpha)_{\alpha \in A}$ in $\B_2(G) \cap \C_0(G)$ such that
\begin{equation}
\label{eq-def-WH-intro}
\norm{u_\alpha}_{\B_2(G)} 
\leq C
\end{equation}
for any $\alpha \in A$ and $u_\alpha \to 1_G$ uniformly on compact subsets of $G$. Here $\B_2(G)$ is 
the space of Herz-Schur multipliers on $G$, defined in Section \ref{sec:preliminaries}. The infimum of the constants $C$ in \eqref{eq-def-WH-intro} is denoted by $\Lambda_{\WH}(G)$. If no such net exists we write $\Lambda_\WH(G) = \infty$. 

Concerning direct products, Knudby observed in \cite[Proposition 5.5 p.~3480]{Knu16} that for any locally compact groups $G$ and $H$ we have the inequality
\begin{equation}
\label{eq-known-submultiplicativity}
\Lambda_{\WH}(G \times H)
\leq
\Lambda_{\WH}(G) \Lambda_{\WH}(H).
\end{equation}
He asked in \cite[Remark 5.6 p.~3480]{Knu16} whether equality holds in \eqref{eq-known-submultiplicativity}. The corresponding equality $\Lambda_{\WA}(G \times H) = \Lambda_{\WA}(G) \Lambda_{\WA}(H)$ for the weak amenability constant $\Lambda_{\mathrm{WA}}$ was proved earlier by Cowling and Haagerup \cite[Corollary 1.5 p.~514]{CoH89}. Recall that a locally compact group $G$ is weakly amenable if there exists a net $(u_\alpha)$ of compactly supported Herz-Schur multipliers on $G$, uniformly bounded in Herz-Schur norm $\norm{\cdot}_{\B_2(G)}$ such that $u_\alpha \to 1$ uniformly on compacts. The least uniform bound on the norms of such nets (if such a bound exists at all) is the weak amenability constant $\Lambda_{\mathrm{WA}}(G)$ of $G$. 

The main result of this paper gives a positive answer to this question.

\begin{thm}
\label{thm-main}
Let $G$ and $H$ be locally compact groups. Then
\begin{equation}
\label{eq-main}
\Lambda_{\WH}(G \times H)
= \Lambda_{\WH}(G) \Lambda_{\WH}(H).
\end{equation}
\end{thm}

It is known \cite[Theorem B p.~6943]{HaK15} that a connected simple Lie group $G$ has the weak Haagerup property if and only if the real rank of $G$ is at most one. For connected simple Lie groups, Haagerup and Knudby proved in \cite[Theorem C p.~6943]{HaK15} that $\Lambda_{\WH}(G)=\Lambda_{\WA}(G)$. In the noncompact real-rank-one cases with finite center, the numerical values of these constants are those computed by Cowling and Haagerup in \cite{CoH89}. See also \cite{Mur08} for a simpler proof. It is proved in \cite[Corollary 3.5 p.~18]{DaG26} that the free product $G*H$ of weakly Haagerup discrete groups $G$ and $H$ with $\Lambda_\WH(G)=\Lambda_\WH(H) = 1$ has weak Haagerup property with $\Lambda_\WH(G*H) = 1$.

\paragraph{Approach of the paper} 
We briefly describe the argument. Recall that by \cite[Proposition 1.10 p.~466]{DCH85} the space $\B_2(G)$ is canonically the dual Banach space of a Banach space $\Q(G)$. For any $q \in \Q(G)$, we define
\begin{equation}
\label{eq-def-p-intro}
p_G(q)
\ov{\mathrm{def}}{=}
\sup\big\{ |\la q,u\ra_{\Q(G),\B_2(G)}| : u \in \B_2(G) \cap \C_0(G), \norm{u}_{\B_2(G)} \leq 1 \big\}.
\end{equation}
Using the weak* characterization of the weak Haagerup property from \cite[Proposition 4.2 p.~3477]{Knu16} and a separation argument, we obtain the polar formula
\begin{equation}
\label{eq-polar-intro}
\Lambda_{\WH}(G)
=
\sup\left\{
\left|\int_G f(s) \d\mu_G(s)\right| : f \in \C_c(G), p_G(f) \leq 1 \right\}.
\end{equation}
The main point is then an estimate for slice maps. If $f \in \C_c(G)$ and $w \in \B_2(G \times H) \cap \C_0(G \times H)$, we define
\begin{equation}
\label{eq-slice-intro}
S_f(w)(t)
\ov{\mathrm{def}}{=}
\int_G f(s)w(s,t) \d\mu_G(s),
\quad t \in H.
\end{equation}
We prove that $S_f(w)$ belongs to the space $\B_2(H) \cap \C_0(H)$ and that
\begin{equation}
\label{eq-slice-estimate-intro}
\norm{S_f(w)}_{\B_2(H)}
\leq
p_G(f) \norm{w}_{\B_2(G \times H)}.
\end{equation}
In fact, the norm of the map $S_f \co \B_2(G \times H) \cap \C_0(G \times H) \to \B_2(H) \cap \C_0(H)$ is equal to $p_G(f)$. The proof of \eqref{eq-slice-estimate-intro} uses the finite matrix characterization of Herz--Schur multipliers. More precisely, a suitable matrix coefficient of a finite restriction of the Herz--Schur kernel associated with $w$ is shown to define an element of $\B_2(G) \cap \C_0(G)$ with the required norm estimate.

Combining \eqref{eq-slice-estimate-intro} with Fubini's theorem gives, for any functions $f \in \C_c(G)$ and $g \in \C_c(H)$,
\begin{equation}
\label{eq-p-tensor-intro}
p_{G\times H}(f \ot g)
=p_G(f)p_H(g).
\end{equation}
The inequality $\Lambda_{\WH}(G) \Lambda_{\WH}(H) \leq \Lambda_{\WH}(G \times H)$ follows from \eqref{eq-polar-intro} and \eqref{eq-p-tensor-intro}, whereas \eqref{eq-known-submultiplicativity} gives the reverse inequality.


\paragraph{Structure of the paper}
The paper is organized as follows. In Section \ref{sec:preliminaries}, we recall the background on Herz--Schur multipliers and the weak Haagerup property. Section \ref{sec:polar} contains the polar description \eqref{eq-polar-intro}. In Section \ref{sec:slices}, we establish the slice estimate \eqref{eq-slice-estimate-intro}. The proof of Theorem \ref{thm-main} is completed in Section \ref{sec:multiplicativity}. We finish with some consequences in Section \ref{sec:consequences}.

\section{Preliminaries}
\label{sec:preliminaries}


\paragraph{Schur multipliers} Let $(\Omega,\mu)$ be a (localizable) measure space. We will use the space $S^\infty_\Omega \ov{\mathrm{def}}{=} S^\infty(\L^2(\Omega))$ of compact operators, its dual $S^1_\Omega$ and the space $\B(\L^2(\Omega))$ of bounded operators on the complex Hilbert space $\L^2(\Omega)$. If $f \in \L^2(\Omega \times \Omega)$, we denote the associated Hilbert-Schmidt operator by
\begin{equation}
\label{Def-de-Kf}
\begin{array}{cccc}
  K_f  \co &  \L^2(\Omega)   &  \longrightarrow   & \L^2(\Omega)   \\
    &   \xi  &  \longmapsto       &  \int_\Omega f(\cdot,y)\xi(y) \d\mu(y)  \\
\end{array}.	
\end{equation}
We say that a function $\varphi \in \L^\infty(\Omega \times \Omega)$ induces a measurable Schur multiplier on $\B(\L^2(\Omega))$ if the map $S^2_\Omega \to \B(\L^2(\Omega))$, $K_{f} \mapsto K_{\varphi f}$ induces a bounded operator from $S^\infty_\Omega$ into $\B(\L^2(\Omega))$. In this case, the operator $S^\infty_\Omega \to \B(\L^2(\Omega))$, $K_{f}\mapsto K_{\varphi f}$ admits by \cite[Lemma A.2.2 p. 360]{BLM04} a unique weak* extension $M_\varphi \co \B(\L^2(\Omega)) \to \B(\L^2(\Omega))$ called the Schur multiplier associated with $\varphi$. 
We refer to the surveys \cite{ToT10} and \cite{Tod15} for more information. See also \cite{Spr04}.

\begin{example} \normalfont
If the set $\Omega=\{1,\ldots,n\}$ is equipped with the counting measure, we can identify the space $\B(\L^2(\Omega))$ with the matrix algebra $\M_n$. Then each operator $K_{f}$ identifies to the matrix $[f(i,j)]$. A Schur multiplier is given by a map $M_\varphi \co \M_n \to \M_n$, $[f(i,j)] \mapsto [\varphi(i,j)f(i,j)]$. 
\end{example}

\paragraph{Herz-Schur multipliers} Let $G$ be a locally compact group equipped with a left Haar measure $\mu_G$. Recall that a continuous function $u \co G \to \mathbb{C}$ is a Herz--Schur multiplier if the kernel $(s,t) \mapsto u(t^{-1}s)$ defines a bounded Schur multiplier $M_{[u(t^{-1}s)]} \co \B(\L^2(G)) \to \B(\L^2(G))$,  on the space $\B(\L^2(G))$ of bounded operators on the complex Hilbert space $\L^2(G)$. The set of Herz-Schur multipliers on $G$ is denoted $\B_2(G)$. It is a unital Banach algebra, when equipped with the Herz-Schur norm $\norm{u}_{\B_2(G)} \ov{\mathrm{def}}{=} \norm{M_{[u(t^{-1}s)]}}_{\B(\L^2(G)) \to \B(\L^2(G))}$. We refer to \cite[p.~466]{BrO08}, \cite{BoF84}, \cite[Chapter 6]{Pis01} and \cite{Tod14} for more information. According to \cite[Proposition 3.1 p.~3475]{Knu16}, we have the following finite matrix characterization. A continuous function $u \co G \to \Cbb$ is a Herz--Schur multiplier if and only if
\begin{equation}
\label{eq-finite-B2-criterion}
\sup\left\{
\norm{M_{[u(s_j^{-1}s_i)]_{i,j=1}^n}}_{\M_n \to \M_n} : n \geq 1, s_1,\ldots,s_n \in G \right\}
<\infty.
\end{equation}
In this case, we have
\begin{equation}
\label{eq-finite-B2}
\norm{u}_{\B_2(G)}
= \sup\left\{
\norm{M_{[u(s_j^{-1}s_i)]_{i,j=1}^n}}_{\M_n \to \M_n} : n \geq 1, s_1,\ldots,s_n \in G
\right\}.
\end{equation}
In particular, for any $u \in \B_2(G)$, we have
\begin{equation}
\label{eq-uniform-vs-B2}
\norm{u}_{\L^\infty(G)}
\leq \norm{u}_{\B_2(G)}.
\end{equation}
Recall that by \cite[Proposition 1.10 p.~466]{DCH85} combined with \cite{BoF84} the Banach space $\B_2(G)$ admits a canonical Banach space predual $\Q(G)$. More precisely, the space $\Q(G)$ is the completion of $\L^1(G)$ for the norm
\begin{equation}
\label{eq-Q-norm}
\norm{f}_{\Q(G)}
\ov{\mathrm{def}}{=}
\sup\left\{
\left|\int_G f(s)u(s) \d\mu_G(s)\right| : u \in \B_2(G), \norm{u}_{\B_2(G)}\leq 1
\right\}.
\end{equation}
Then $\Q(G)^* = \B_2(G)$ isometrically, where the duality between $f \in \L^1(G)$ and $u \in \B_2(G)$ is given by
\begin{equation}
\label{eq-duality-Q-B2}
\la f,u \ra_{\Q(G),\B_2(G)}
=
\int_G f(s)u(s) \d\mu_G(s).
\end{equation}
Furthermore, \eqref{eq-uniform-vs-B2} implies that
\begin{equation}
\label{eq-Q-L1}
\norm{f}_{\Q(G)}\leq \norm{f}_{\L^1(G)},
\quad f \in \L^1(G).
\end{equation}
Consequently, the subspace $\C_c(G)$ is dense in $\Q(G)$.

Let $H$ be a locally compact group. Whenever $G$ and $H$ are equipped with left Haar measures
$\mu_G$ and $\mu_H$, respectively, we equip the direct product $G\times H$ with the product left Haar measure $\mu_G \ot \mu_H$. We record two elementary facts which will also be used. If $u \in \B_2(G)$ and $v \in \B_2(H)$, then the function $u \ot v$ on $G \times H$ defined by
\begin{equation*}
(u \ot v)(s,t)
\ov{\mathrm{def}}{=} u(s)v(t), \quad s \in G,t \in H,
\end{equation*}
is a Herz--Schur multiplier and
\begin{equation}
\label{eq-product-B2}
\norm{u \ot v}_{\B_2(G \times H)}
\leq
\norm{u}_{\B_2(G)} \norm{v}_{\B_2(H)}.
\end{equation}
Indeed, this follows immediately from the Hilbert space factorization of Herz--Schur multipliers. See also \cite[Corollary 5.4.11 p.~185]{KaL18}. Moreover, if $u \in \C_0(G)$ and $v \in \C_0(H)$, then the function $u \ot v$ also belongs to the space $\C_0(G\times H)$.

\paragraph{Weak Haagerup property} The following weak* formulation is \cite[Proposition 4.2 p.~3477]{Knu16}. It will be used in an essential way.

\begin{prop}
\label{prop-weakstar-WH}
Let $G$ be a locally compact group. If $C > 0$ then $\Lambda_{\WH}(G) \leq C$ if and only if there exists a net $(u_\alpha)_{\alpha \in A}$ of functions in the space $\B_2(G) \cap \C_0(G)$ such that
\begin{equation*}
\norm{u_\alpha}_{\B_2(G)}
\leq C,
\quad \alpha\in A,
\end{equation*}
and $u_\alpha \to 1_G$ in the topology $\sigma(\B_2(G),\Q(G))$.
\end{prop}

If $H$ is a closed subgroup of $G$, then according to \cite[Proposition 5.2 p.~3479]{Knu16} the restriction map gives a contraction from the space $\B_2(G)$ into $\B_2(H)$ and maps the space $\C_0(G)$ into $\C_0(H)$ and we have
\begin{equation}
\label{eq-subgroup-WH}
\Lambda_{\WH}(H)
\leq \Lambda_{\WH}(G).
\end{equation}

\paragraph{Polars} Let $X,Y$ be two complex vector spaces in duality. If $A$ is a subset of $X$ we define as in \cite[Definition 1, II.64]{Bou87} the polar
\[
A^\circ
\ov{\mathrm{def}}{=} \{y \in Y : \operatorname{Re}(\la x,y \ra_{X,Y}) \geq -1 \text{ for all } x \in A \}.
\]
This is convex, being an intersection of the convex sets $\{\varphi \in X^*:\operatorname{Re}(\la x,y \ra_{X,Y})\geq -1\}$, for all $x \in A$. It also contains 0. In our applications usually $A$ has the property that $\alpha x \in A$ whenever $|\alpha| \leq 1$ and $x \in A$, and in this case it is easy to see \cite[II.64]{Bou87} that
\begin{equation}
\label{polar-prime}
A^\circ
=\{ y \in Y : |\la x,y \ra_{X,Y}| \leq 1 \text{ for all } x \in A \}.
\end{equation}
The bipolar of $A$ is $A^{\circ\circ}$. We have the following result, see \cite[Theorem 1.5 p.~126]{ScW99} and \cite[Theorem 1 II.44]{Bou87}.

\begin{thm}[bipolar theorem]
\label{th-bipolar}
Let $X,Y$ be two complex vector spaces in duality. If $A$ is a subset of $X$ then the bipolar $A^{\circ\circ}$ of $A$ is the closed convex hull (for $\sigma(X,Y)$) of the set $A \cup \{0\}$.
\end{thm}

\section{A polar description of the weak Haagerup constant}
\label{sec:polar}

We equip the space $\B_2(G) \cap \C_0(G)$ with the norm inherited from the Banach space $\B_2(G)$. Note that the subspace $\B_2(G) \cap \C_0(G)$ is closed in $\B_2(G)$. Indeed, convergence in $\B_2(G)$ implies uniform convergence by \eqref{eq-uniform-vs-B2}, and $\C_0(G)$ is closed for the uniform norm.


For any function $f \in \L^1(G)$, we define
\begin{equation}
\label{eq-def-pG}
p_G(f)
\ov{\mathrm{def}}{=}
\sup\left\{
\left|\int_G f(s)u(s)\d\mu_G(s)\right| :
u \in \B_2(G) \cap \C_0(G),
\norm{u}_{\B_2(G)} \leq 1
\right\}.
\end{equation}
We first extend the seminorm $p_G$ from $\L^1(G)$ to $\Q(G)$. For any functions $f,g \in \L^1(G)$, we have
\begin{equation*}
|p_G(f)-p_G(g)|
\leq
p_G(f-g)
\ov{\eqref{eq-def-pG} \eqref{eq-Q-norm}}{\leq} 
\norm{f-g}_{\Q(G)}.
\end{equation*}
Consequently, since the space $\Q(G)$ is the completion of $\L^1(G)$ for the norm $\norm{\cdot}_{\Q(G)}$, the seminorm $p_G$ admits a unique continuous
extension to $\Q(G)$, which we still denote by $p_G$. This extension satisfies
\begin{equation}
\label{eq-pG-Q-extension}
|p_G(q)-p_G(q')|
\leq
\norm{q-q'}_{\Q(G)},
\quad q,q'\in\Q(G).
\end{equation}

We record a description of this extension which will be useful below. 

\begin{lemma}
For any $q \in \Q(G)$, we have
\begin{equation}
\label{eq-pG-Q-abstract}
p_G(q)
=\sup\left\{
|\la q,u \ra_{\Q(G),\B_2(G)}| : u \in \B_2(G) \cap \C_0(G), \norm{u}_{\B_2(G)}\leq 1 \right\}.
\end{equation}
\end{lemma}

\begin{proof}
Let $(f_n)$ be a sequence of functions in $\L^1(G)$ such that $\norm{f_n-q}_{\Q(G)} \to 0$. Since $\Q(G)^* = \B_2(G)$ isometrically, we have
\begin{align*}
\MoveEqLeft
\left|
\sup_{\substack{u \in \B_2(G) \cap \C_0(G)\\ \norm{u}_{\B_2(G)}\leq1}}
|\la q,u\ra|
-
p_G(f_n)
\right| 
\ov{\eqref{eq-duality-Q-B2} \eqref{eq-def-pG}}{\leq} 
\sup_{\substack{u \in \B_2(G) \cap \C_0(G)\\ \norm{u}_{\B_2(G)}\leq1}}
\big| |\la q,u\ra| - |\la f_n,u \ra_{\Q(G),\B_2(G)}| \big| \\
&\leq \sup_{\substack{u \in \B_2(G) \cap \C_0(G)\\ \norm{u}_{\B_2(G)} \leq 1}}
|\la q,u\ra_{\Q(G),\B_2(G)}-\la f_n,u \ra_{\Q(G),\B_2(G)}| \\
&= \sup_{\substack{u \in \B_2(G) \cap \C_0(G)\\ \norm{u}_{\B_2(G)} \leq 1}}
|\la q-f_n,u \ra_{\Q(G),\B_2(G)}|
\leq
\norm{q-f_n}_{\Q(G)}.
\end{align*}
Here, in the first inequality we used the elementary estimate
\begin{equation*}
\left|\sup_{i\in I} a_i-\sup_{i\in I}b_i\right|
\leq
\sup_{i\in I}|a_i-b_i|,
\end{equation*}
whereas the second inequality follows from
$\big||z|-|w|\big|\leq|z-w|$. The right-hand side converges to $0$, whereas $p_G(f_n)\to p_G(q)$ by \eqref{eq-pG-Q-extension}. This proves
\eqref{eq-pG-Q-abstract}.
\end{proof}


Since $1_G \in \B_2(G)$ and $\norm{1_G}_{\B_2(G)}=1$, the linear
functional
\begin{equation*}
f \mapsto \int_G f(s)\d\mu_G(s),
\quad f\in\L^1(G),
\end{equation*}
is continuous for the $\Q(G)$-norm. Hence it admits a unique continuous extension to $\Q(G)$. Under the duality $\Q(G)^*=\B_2(G)$, this extension is the functional $\Q(G) \to \mathbb{C}$, $q \mapsto \la q,1_G\ra_{\Q(G),\B_2(G)}$.

The following proposition is the first ingredient of the proof of Theorem \ref{thm-main}.

\begin{prop}
\label{prop-polar}
Let $G$ be a locally compact group. Then
\begin{equation}
\label{eq-polar}
\Lambda_{\WH}(G)
=
\sup\left\{
\left|\int_G f(s) \d\mu_G(s)\right| : f\in \C_c(G), p_G(f)\leq 1
\right\}.
\end{equation}
The supremum in \eqref{eq-polar} is allowed to be infinite.
\end{prop}

\begin{proof}
Let $\mathrm{Ball}_{\B_2(G) \cap \C_0(G)}$ denote the closed unit ball of the space $\B_2(G) \cap \C_0(G)$ equipped with the norm $\norm{\cdot}_{\B_2(G)}$. We put
\begin{equation*}
K_C
\ov{\mathrm{def}}{=}
C\mathrm{Ball}_{\B_2(G)\cap\C_0(G)}.
\end{equation*}
According to Proposition \ref{prop-weakstar-WH}, for any constant $C > 0$, we have
\begin{equation}
\label{eq-closure-polar}
\Lambda_{\WH}(G)\leq C
\quad\Longleftrightarrow\quad
1_G\in
\overline{K_C}^{\sigma(\B_2(G),\Q(G))}.
\end{equation}
We equip $\B_2(G)$ with the topology $\sigma(\B_2(G),\Q(G))$. Since $\Q(G)$ is a predual of $\B_2(G)$, the continuous dual of $\big(\B_2(G),\sigma(\B_2(G),\Q(G))\big)$ is precisely $\Q(G)$. 
The set $K_C$ is absolutely convex and contains $0$. Hence, by the bipolar theorem (Theorem \ref{th-bipolar}), we have
\begin{equation}
\label{eq-bipolar-KC}
K_C^{\circ \circ}
=\overline{K_C}^{\sigma(\B_2(G),\Q(G))}.
\end{equation}
Now, we compute the polar of the subset $K_C$. By \eqref{eq-pG-Q-abstract}, for any $q \in \Q(G)$, we have
\begin{align*}
q \in K_C^\circ
\ov{\eqref{polar-prime}}{\Longleftrightarrow} 
\sup_{u \in K_C}|\la q,u\ra|
\leq 1
\Longleftrightarrow
C
\sup_{\substack{u \in \B_2(G) \cap \C_0(G)\\
\norm{u}_{\B_2(G)} \leq 1}}
|\la q,u\ra|
\leq1
\ov{\eqref{eq-pG-Q-abstract}}{\Longleftrightarrow} Cp_G(q) \leq 1.
\end{align*}
Thus $K_C^\circ = \{q \in \Q(G) : Cp_G(q) \leq 1 \}$. It follows from this equality and from \eqref{eq-bipolar-KC} that
\begin{align}
1_G \in \overline{K_C}^{\sigma(\B_2(G),\Q(G))}
\quad\ov{\eqref{eq-bipolar-KC}\eqref{polar-prime}}{\Longleftrightarrow} \quad
|\la q,1_G\ra|
\leq 1
\quad
\text{for any }q\in\Q(G)
\text{ satisfying }Cp_G(q)
\leq1.
\label{eq-bipolar-condition}
\end{align}
We claim that the condition on the right-hand side of \eqref{eq-bipolar-condition} is equivalent to
\begin{equation}
\label{eq-polar-inequality-Q}
|\la q,1_G\ra|
\leq
Cp_G(q),
\quad q\in\Q(G).
\end{equation}
Indeed, suppose first that the condition in \eqref{eq-bipolar-condition} holds and let $q\in\Q(G)$. If $p_G(q)>0$, then
\begin{equation*}
p_G\left(\frac{q}{Cp_G(q)}\right)
=
\frac{1}{C},
\end{equation*}
and hence \eqref{eq-bipolar-condition} gives
\begin{equation*}
\frac{|\la q,1_G\ra|}{Cp_G(q)}
\leq1.
\end{equation*}
This proves \eqref{eq-polar-inequality-Q} in this case.

Now, suppose that $p_G(q)=0$. For any $t > 0$, we have $Cp_G(tq)=0\leq1$, and hence \eqref{eq-bipolar-condition} gives
\begin{equation*}
t|\la q,1_G\ra|
=
|\la tq,1_G\ra|
\leq1.
\end{equation*}
Letting $t$ tend to infinity, we obtain $\la q,1_G\ra=0$. Thus \eqref{eq-polar-inequality-Q} also holds when $p_G(q)=0$. The converse implication follows immediately from \eqref{eq-polar-inequality-Q}.

Combining \eqref{eq-closure-polar} and \eqref{eq-bipolar-condition}, we
therefore obtain
\begin{equation}
\label{eq-polar-C-Q}
\Lambda_{\WH}(G)\leq C
\quad\Longleftrightarrow\quad
|\la q,1_G\ra|
\leq Cp_G(q)
\quad \text{for any }q \in \Q(G).
\end{equation}
We next show that the inequality on the right-hand side of \eqref{eq-polar-C-Q} may be tested on the dense subspace $\C_c(G)$. Indeed, the map $\Q(G) \to \mathbb{C}$, $q \mapsto \la q,1_G\ra$ is continuous, since $\norm{1_G}_{\B_2(G)}=1$, whereas the map $q \mapsto p_G(q)$ is continuous by \eqref{eq-pG-Q-extension}. Consequently, since the subspace $\C_c(G)$ is dense in the Banach space $\Q(G)$, the condition \eqref{eq-polar-inequality-Q} is equivalent to
\begin{equation}
\label{eq-polar-inequality-Cc}
|\la f,1_G\ra|
\leq
Cp_G(f),
\quad f \in \C_c(G).
\end{equation}
More explicitly, if \eqref{eq-polar-inequality-Cc} holds, let $q \in \Q(G)$ and choose a sequence $(f_n)$ of functions in $\C_c(G)$ such that $\norm{f_n-q}_{\Q(G)} \to 0$. Then $
\la f_n,1_G\ra
\to
\la q,1_G\ra$ and, by \eqref{eq-pG-Q-extension},
\begin{equation*}
p_G(f_n) \to p_G(q).
\end{equation*}
Passing to the limit in
\begin{equation*}
|\la f_n,1_G\ra|
\ov{\eqref{eq-polar-inequality-Cc}}{\leq}  Cp_G(f_n)
\end{equation*}
gives \eqref{eq-polar-inequality-Q}. The converse implication follows by restriction to $\C_c(G)$. For any function $f \in \C_c(G)$, the duality formula \eqref{eq-duality-Q-B2} gives
\begin{equation*}
\la f,1_G\ra
\ov{\eqref{eq-duality-Q-B2}}{=} 
\int_G f(s) \d\mu_G(s).
\end{equation*}
Hence \eqref{eq-polar-C-Q} is equivalent to
\begin{equation}
\label{eq-polar-C-Cc}
\Lambda_{\WH}(G)\leq C
\quad\Longleftrightarrow\quad
\left|\int_G f(s)\d\mu_G(s)\right|
\leq
Cp_G(f)\quad \text{for any } f \in \C_c(G).
\end{equation}
Let
\begin{equation*}
D
\ov{\mathrm{def}}{=}
\sup\left\{
\left|\int_G f(s)\d\mu_G(s)\right|:
f\in\C_c(G), p_G(f)\leq1
\right\}.
\end{equation*}
We finally observe that, for any constant $C > 0$, the condition on the right-hand side of \eqref{eq-polar-C-Cc} is equivalent to $D \leq C$. One implication is immediate. Conversely, suppose that $D \leq C$ and let $f \in \C_c(G)$. If $p_G(f) > 0$, applying the definition of $D$ to $\frac{1}{p_G(f)}f$ gives
\begin{equation*}
\left|\int_G f(s) \d\mu_G(s)\right|
\leq
D p_G(f)
\leq C p_G(f).
\end{equation*}
If $p_G(f)=0$, then $p_G(tf) = 0 \leq 1$ for any $t > 0$. Hence
\begin{equation*}
t\left|\int_G f(s) \d\mu_G(s)\right|
= \left|\int_G tf(s) \d\mu_G(s)\right|
\leq D
\end{equation*}
for any $t>0$. Letting $t$ tend to infinity, we obtain
\begin{equation*}
\int_G f(s) \d\mu_G(s)
=0.
\end{equation*}
Thus the right-hand side of \eqref{eq-polar-C-Cc} holds. We have therefore proved that, for any $C>0$,
\begin{equation*}
\Lambda_{\WH}(G)
\leq C \quad\Longleftrightarrow\quad
D\leq C.
\end{equation*}
It follows that $\Lambda_{\WH}(G)=D$, with the equality understood in
$[0,\infty]$. This proves \eqref{eq-polar}.
\end{proof}

\section{Slice maps of Herz--Schur multipliers}
\label{sec:slices}

We first prove a finite matrix lemma. It is the main technical point of the paper.

\begin{lemma}
\label{lemma-matrix-slice}
Let $G$ and $H$ be locally compact groups and let $w \in \B_2(G \times H)$. Fix $n \geq 1$, $h_1,\ldots,h_n \in H$, $A=[a_{ij}] \in \M_n$ and $\xi=(\xi_j)_{j=1}^n$, $\eta=(\eta_i)_{i=1}^n \in \Cbb^n$. Define the function $u \co G \to \Cbb$ by
\begin{equation}
\label{eq-def-u-slice}
u(s)
\ov{\mathrm{def}}{=}
\sum_{i,j=1}^n
\overline{\eta_i}a_{ij}\xi_j
w(s,h_j^{-1}h_i),
\quad s \in G.
\end{equation}
Then $u\in \mathrm{B}_2(G)$ and
\begin{equation}
\label{eq-u-slice-bound}
\norm{u}_{\B_2(G)}
\leq
\norm{w}_{\B_2(G \times H)} \norm{A}\norm{\xi}_{\ell^2_n} \norm{\eta}_{\ell^2_n}.
\end{equation}
If, in addition, $w\in \C_0(G \times H)$, then $u \in \C_0(G)$.
\end{lemma}

\begin{proof}
Fix $m \geq 1$, $s_1,\ldots,s_m \in G$ and $B=[b_{kl}] \in \M_m$. Consider the matrix $K \in \M_{mn}$ indexed by pairs $(k,i)$ and $(l,j)$ and defined by
\begin{equation}
\label{eq-def-K}
K_{(k,i),(l,j)}
=
w(s_l^{-1}s_k,h_j^{-1}h_i), \quad 1 \leq k,l \leq m, 1 \leq i,j \leq n.
\end{equation}
The points $(s_k,h_i)$ belong to the group $G \times H$ and $(s_l,h_j)^{-1}(s_k,h_i)=(s_l^{-1} s_k,h_j^{-1}h_i)$. The norm of the Schur multiplier $M_K \co \M_{mn}(\Cbb) \to \M_{mn}(\Cbb)$ can be estimated by
\begin{equation}
\label{eq-K-Schur}
\norm{M_K}_{\M_{mn}(\Cbb) \to \M_{mn}(\Cbb)} 
\ov{\eqref{eq-finite-B2}}{\leq}
\norm{w}_{\B_2(G \times H)}.
\end{equation}
Let $V_\xi,V_\eta \co \Cbb^m \to \Cbb^m \ot \Cbb^n$ be the linear operators defined by
\begin{equation}
\label{def-V-xi}
V_\xi(z)
\ov{\mathrm{def}}{=} z \ot \xi
\quad \text{and} \quad
V_\eta(z)
\ov{\mathrm{def}}{=} z \ot \eta, \quad z \in \Cbb^m,
\end{equation}
where $\xi,\eta \in \Cbb^n$. Then $
\norm{V_\xi}_{\ell^2_n \to \ell^2_n \ot_2 \ell^2_n}
=\norm{\xi}_{\ell^2_n}$ and $\norm{V_\eta}_{\ell^2_n \to \ell^2_n \ot_2 \ell^2_n}
=\norm{\eta}_{\ell^2_n}$. 
We identify $\M_{mn}$ with the space $\B(\Cbb^m \ot \Cbb^n)$ with respect to the canonical orthonormal basis $(e_k \ot f_i)_{1 \leq k \leq m,1 \leq i \leq n}$, where $(e_k)_{k=1}^m$ and $(f_i)_{i=1}^n$ denote the canonical bases of the vector spaces $\Cbb^m$ and $\Cbb^n$, respectively. For any $1 \leq l \leq m$, observe that
\begin{equation}
\label{fin-35}
V_\xi e_l
 \ov{\eqref{def-V-xi}}{=} e_l \ot \xi
=\sum_{j=1}^n \xi_j e_l \ot f_j, \quad \xi \in \ell^2_n.
\end{equation}
Moreover, the matrix of $B\ot A$ with respect to the previous basis is
\[
(B \ot A)_{(k,i),(l,j)}
=b_{kl}a_{ij}.
\]
Hence, by the definition \eqref{eq-def-K} of the matrix $K$, the matrix of the Schur product $M_K(B\ot A)$ is given by
\begin{equation}
\label{fin-32}
\big[M_K(B \ot A)\big]_{(k,i),(l,j)}
=
K_{(k,i),(l,j)}b_{kl}a_{ij}.
\end{equation}
Consequently, for any $1 \leq l \leq m$, we have
\begin{align}
\label{fin-2}
M_K(B\ot A)V_\xi e_l
\ov{\eqref{fin-32}\eqref{fin-35}}{=} 
\sum_{k=1}^m\sum_{i,j=1}^n K_{(k,i),(l,j)}b_{kl}a_{ij}\xi_j e_k\ot f_i. 
\end{align}
On the other hand, the definition of $V_\eta$ gives
\begin{equation}
\label{fin-1}
V_\eta^*(e_k\ot f_i)
=\overline{\eta_i}e_k.
\end{equation}
It follows that
\begin{align*}
\MoveEqLeft
V_\eta^*M_K(B\ot A)V_\xi e_l
\ov{\eqref{fin-2}}{=}V_\eta^*\bigg(\sum_{k=1}^m\sum_{i,j=1}^n K_{(k,i),(l,j)}b_{kl}a_{ij}\xi_j e_k\ot f_i\bigg) \\
&\ov{\eqref{fin-1}}{=}
\sum_{k=1}^m
\bigg(
\sum_{i,j=1}^n
\overline{\eta_i}
K_{(k,i),(l,j)}
b_{kl}a_{ij}\xi_j
\bigg)e_k\\
&\ov{\eqref{eq-def-K}}{=}
\sum_{k=1}^m
b_{kl}
\bigg(
\sum_{i,j=1}^n
\overline{\eta_i}a_{ij}\xi_j
w(s_l^{-1}s_k,h_j^{-1}h_i)
\bigg)e_k
\ov{\eqref{eq-def-u-slice}}{=}
\sum_{k=1}^m b_{kl}u(s_l^{-1}s_k)e_k.
\end{align*}
Since this equality holds for any $1\leq l\leq m$, we obtain
\begin{equation}
\label{eq-compression-identity}
V_\eta^*M_K(B\ot A)V_\xi
=
[b_{kl}u(s_l^{-1}s_k)]_{k,l=1}^m.
\end{equation}
Consequently, since $\norm{B \ot A} = \norm{B}\norm{A}$, we have
\begin{align*}
\MoveEqLeft
\norm{[b_{kl}u(s_l^{-1}s_k)]_{k,l=1}^m}
\ov{\eqref{eq-compression-identity}}{=} \norm{V_\eta^*M_K(B\ot A)V_\xi}
\leq
\norm{V_\eta} \norm{M_K} \norm{B \ot A} \norm{V_\xi} \\
&\ov{\eqref{eq-K-Schur}}{\leq}
\norm{w}_{\B_2(G \times H)}\norm{B} \norm{A}\norm{\xi}_{\ell^2_n}\norm{\eta}_{\ell^2_n}.
\end{align*}
Since the matrix $B \in \M_m(\mathbb{C})$ is arbitrary, we deduce that
\begin{equation*}
\norm{M_{[u(s_l^{-1}s_k)]_{k,l=1}^m}}_{\M_m \to \M_m}
\leq
\norm{w}_{\B_2(G \times H)} \norm{A} \norm{\xi}_{\ell^2_n} \norm{\eta}_{\ell^2_n}.
\end{equation*}
Since $w$ is continuous, formula \eqref{eq-def-u-slice} shows that $u \co G \to \Cbb$ is continuous. Since $m$, $s_1,\ldots,s_m$ and $B \in \M_m$ are arbitrary, the finite matrix characterization \eqref{eq-finite-B2-criterion} of Herz--Schur multipliers now implies that $u \in \B_2(G)$ and gives \eqref{eq-u-slice-bound}.

Suppose finally that $w \in \C_0(G\times H)$. Fix $t\in H$. The function $s \mapsto w(s,t)$ is continuous. For any $\epsi > 0$, the set
\begin{equation*}
\{s\in G:|w(s,t)|\geq\epsi\}
\end{equation*}
is the projection onto $G$ of the compact set
\begin{equation*}
\{(s,r) \in G \times H : |w(s,r)|\geq\epsi\} \cap (G\times\{t\}).
\end{equation*}
Hence $w(\cdot,t) \in \C_0(G)$. Formula \eqref{eq-def-u-slice} expresses $u$ as a finite linear combination of such functions. Consequently, the function $u$ belongs to the space $\C_0(G)$.
\end{proof}

Now, we can prove the exact estimate for slice maps.

\begin{prop}
\label{prop-slice}
Let $G$ and $H$ be locally compact groups and let $f \in \C_c(G)$. For $w \in \B_2(G \times H) \cap \C_0(G \times H)$, define
\begin{equation}
\label{eq-def-Sf}
S_f(w)(t)
\ov{\mathrm{def}}{=}
\int_G f(s)w(s,t) \d\mu_G(s),
\quad t \in H.
\end{equation}
Then the function $S_f(w)$ belongs to the space $\B_2(H) \cap \C_0(H)$ and
\begin{equation}
\label{eq-Sf-bound}
\norm{S_f(w)}_{\B_2(H)}
\leq p_G(f) \norm{w}_{\B_2(G \times H)}.
\end{equation}
Moreover, we have
\begin{equation}
\label{eq-Sf-norm}
\norm{S_f}_{\B_2(G \times H) \cap \C_0(G \times H) \to \B_2(H) \cap \C_0(H)}
=p_G(f).
\end{equation}
\end{prop}

\begin{proof}
We first check that the function in \eqref{eq-def-Sf} belongs to the space $\C_0(H)$. Let $K \ov{\mathrm{def}}{=}\supp f$. Since $K$ is compact, the map $H \to \C(K)$, $t \mapsto w(\cdot,t)|_K$ is continuous for the uniform norm. Indeed, let $t_0 \in H$ and $\epsi > 0$. By the joint continuity of $w$ and the compactness of $K$, there exists a neighbourhood $V$ of
$t_0$ such that
\begin{equation*}
\sup_{s \in K}|w(s,t)-w(s,t_0)| < \epsi,
\quad  t\in V.
\end{equation*}
Thus the map $t \mapsto w(\cdot,t)|_K$ is continuous from $H$ into $\C(K)$ equipped with the uniform norm. Hence $S_f(w)$ is continuous.

Now, suppose first that $f \neq 0$. Let $\epsi > 0$. Since $w \in \C_0(G \times H)$, there exists a compact subset $L$ of $G \times H$ such that
\begin{equation*}
|w(s,t)|
<\frac{\epsi}{\norm{f}_{\L^1(G)}}
\end{equation*}
whenever $(s,t) \notin L$. If $t \notin \operatorname{pr}_H(L)$, then $(s,t) \notin L$ for any $s\in G$, and therefore
\begin{equation*}
|S_f(w)(t)|
\ov{\eqref{eq-def-Sf}}{=} \left|\int_G f(s)w(s,t) \d\mu_G(s)\right|
\leq \int_G |f(s)||w(s,t)| \d\mu_G(s)
< \epsi.
\end{equation*}
Since $\operatorname{pr}_H(L)$ is compact, this proves that $S_f(w)$ belongs to the space $\C_0(H)$. The case $f=0$ is immediate.

It remains to prove the Herz--Schur estimate. Fix $n \geq 1$, $h_1,\ldots,h_n \in H$ and $A=[a_{ij}]\in \M_n(\Cbb)$. Let $\xi=(\xi_j)$ and $\eta=(\eta_i)$ be vectors in $\Cbb^n$. Define $u$ by \eqref{eq-def-u-slice}. By Lemma \ref{lemma-matrix-slice}, we have $u \in \B_2(G) \cap \C_0(G)$ and
\begin{equation}
\label{eq-u-bound-use}
\norm{u}_{\B_2(G)}
\ov{\eqref{eq-u-slice-bound}}{\leq}
\norm{w}_{\B_2(G \times H)} \norm{A} \norm{\xi}_{\ell^2_n} \norm{\eta}_{\ell^2_n}.
\end{equation}
Using the fact that the sum in \eqref{eq-def-u-slice} is finite, we obtain
\begin{align*}
\MoveEqLeft
\left|
\sum_{i,j=1}^n
\overline{\eta_i}a_{ij}\xi_j S_f(w)(h_j^{-1}h_i)
\right|
\ov{\eqref{eq-def-Sf}}{=} \left|\sum_{i,j=1}^n \overline{\eta_i}a_{ij}\xi_j \int_G f(s)w(s,(h_j^{-1}h_i)) \d\mu_G(s)\right| \\
&= \left|\int_G f(s)\sum_{i,j=1}^n
\overline{\eta_i}a_{ij}\xi_jw(s,(h_j^{-1}h_i)) \d\mu_G(s)\right|
\ov{\eqref{eq-def-u-slice}}{=} \left|\int_G f(s)u(s)\d\mu_G(s)\right| \\
&\ov{\eqref{eq-def-pG}}{\leq}
p_G(f)\norm{u}_{\B_2(G)}
\ov{\eqref{eq-u-bound-use}}{\leq}
p_G(f) \norm{w}_{\B_2(G \times H)} \norm{A} \norm{\xi}_{\ell^2_n} \norm{\eta}_{\ell^2_n}.
\end{align*}
The expression on the first line is the absolute value of the scalar matrix coefficient between $\xi$ and $\eta$ of
\begin{equation*}
M_{[S_f(w)(h_j^{-1}h_i)]_{i,j=1}^n}(A).
\end{equation*}
Taking the supremum over unit vectors $\xi$ and $\eta$, we get
\begin{equation*}
\norm{M_{[S_f(w)(h_j^{-1}h_i)]_{i,j=1}^n}(A)}
\leq
p_G(f) \norm{w}_{\mathrm{B}_2(G \times H)} \norm{A}.
\end{equation*}
Taking the supremum over $A \in \M_n(\Cbb)$ with $\norm{A}\leq 1$, and then using \eqref{eq-finite-B2}, yields \eqref{eq-Sf-bound}.

We finish with the reverse estimate in \eqref{eq-Sf-norm}. Choose a nonzero vector $\zeta \in \C_c(H)$ and define
\begin{equation*}
v_0(t)
\ov{\mathrm{def}}{=}
\frac{\la \lambda_t\zeta,\zeta\ra}{\norm{\zeta}_{\L^2(H)}^2},
\quad t\in H,
\end{equation*}
where $\lambda$ denotes the left regular representation of the group $H$. Then $v_0$ belongs to the Fourier algebra $\mathrm{A}(H)$. Hence $v_0 \in \C_0(H)$ by \cite[Corollary 2.3.5 p.~52]{KaL18}. Moreover, $v_0$ is positive definite, $v_0(e_H) = 1$ and
\begin{equation}
\label{eq-v0-norm}
\norm{v_0}_{\B_2(H)} 
= 1.
\end{equation}
Indeed, we have
\begin{equation*}
v_0(t^{-1}s)
=
\left\la
\frac{\lambda_s\zeta}{\norm{\zeta}_2},
\frac{\lambda_t\zeta}{\norm{\zeta}_2}
\right\ra,
\quad s,t\in H.
\end{equation*}
Hence the Hilbert space factorization of Herz--Schur multipliers gives $\norm{v_0}_{\B_2(H)} \leq 1$. On the other hand, $v_0(e_H)=1$, so \eqref{eq-uniform-vs-B2} gives $\norm{v_0}_{\B_2(H)} \geq 1$.

Let $u \in \B_2(G) \cap \C_0(G)$ with $\norm{u}_{\B_2(G)} \leq 1$. By \eqref{eq-product-B2}, the function $w=u\ot v_0$ belongs to the space $\B_2(G \times H) \cap \C_0(G \times H)$ and satisfies $\norm{w}_{\mathrm{B}_2(G \times H)}\leq 1$. Furthermore, we have
\begin{equation*}
S_f(w)(t)
\ov{\mathrm{def}}{=}
\int_G f(s)w(s,t) \d\mu_G(s)
=\int_G f(s) u(s) v_0(t) \d\mu_G(s)
=\int_G u(s)f(s) \d\mu_G(s)v_0(t).
\end{equation*}
Thus
\begin{equation*}
S_f(w)
=\la f,u \ra v_0.
\end{equation*}
Hence, by \eqref{eq-v0-norm}, we have
\begin{equation*}
\norm{S_f}_{\B_2(G \times H) \cap \C_0(G \times H) \to \B_2(H) \cap \C_0(H)}
\geq |\la f,u\ra|.
\end{equation*}
Taking the supremum over $u \in \B_2(G) \cap \C_0(G)$ with $\norm{u}_{\B_2(G)} \leq 1$ gives
\begin{equation*}
\norm{S_f}_{\B_2(G \times H) \cap \C_0(G \times H) \to \B_2(H) \cap \C_0(H)}
\geq p_G(f).
\end{equation*}
Together with \eqref{eq-Sf-bound}, this proves \eqref{eq-Sf-norm}.
\end{proof}

\section{Multiplicativity}
\label{sec:multiplicativity}

The slice estimate gives a multiplicativity formula for the seminorms introduced in \eqref{eq-def-pG}.

\begin{prop}
\label{prop-p-tensor}
Let $G$ and $H$ be locally compact groups. For any functions $f \in \C_c(G)$ and $g \in \C_c(H)$, we have
\begin{equation}
\label{eq-p-tensor}
p_{G \times H}(f \ot g)
= p_G(f) p_H(g),
\end{equation}
where $(f \ot g)(s,t) = f(s)g(t)$.
\end{prop}

\begin{proof}
Let $w \in \B_2(G \times H) \cap \C_0(G \times H)$ with $\norm{w}_{\B_2(G \times H)} \leq 1$. Since the functions $f$ and $g$ have compact support and $w$ is bounded, the function
\begin{equation*}
(s,t) \mapsto f(s)g(t)w(s,t)
\end{equation*}
is integrable on the product $G \times H$. Hence Fubini's theorem and Proposition \ref{prop-slice} give
\begin{align*}
\MoveEqLeft
|\la f \ot g,w \ra| 
=\left| \int_{G \times H} g(t) f(s)w(s,t) \d(\mu_G \ot \mu_H)(s,t) \right|\\
&=\left| \int_H g(t) \int_G f(s)w(s,t) \d\mu_G(s) \d\mu_H(t) \right|
\ov{\eqref{eq-def-Sf}}{=}
\left|
\int_H g(t)S_f(w)(t) \d\mu_H(t)
\right| \\
&\ov{\eqref{eq-def-pG}}{\leq}
p_H(g)\norm{S_f(w)}_{\B_2(H)}
\ov{\eqref{eq-Sf-bound}}{\leq}
p_G(f)p_H(g)\norm{w}_{\B_2(G \times H)}
\leq p_G(f)p_H(g).
\end{align*}
Taking the supremum over the unit ball of $\B_2(G \times H) \cap \C_0(G \times H)$ gives
\begin{equation}
\label{eq-p-tensor-upper}
p_{G \times H}(f \ot g)
\leq p_G(f) p_H(g).
\end{equation}

For the converse estimate, consider some functions $u \in \B_2(G) \cap \C_0(G)$ and $v \in \B_2(H) \cap \C_0(H)$ such that
\begin{equation*}
\norm{u}_{\B_2(G)} \leq 1,
\quad \text{and} \quad
\norm{v}_{\B_2(H)} \leq 1.
\end{equation*}
By \eqref{eq-product-B2}, the function $u \ot v \co G \times H \to \mathbb{C}$ belongs to the unit ball of the space $\B_2(G \times H) \cap \C_0(G \times H)$. Therefore,
\begin{align*}
\MoveEqLeft
p_{G\times H}(f \ot g)
\ov{\eqref{eq-def-pG}}{\geq} 
|\la f\ot g,u \ot v\ra|
=
|\la f,u \ra| |\la g,v \ra|,
\end{align*}
where the last equality follows from Fubini's theorem. Taking independently the supremum over $u$ and $v$ proves the reverse inequality in \eqref{eq-p-tensor-upper}. This proves \eqref{eq-p-tensor}.
\end{proof}

Now, we can prove the main theorem.

\begin{thm}
\label{thm-main-bis}
Let $G$ and $H$ be locally compact groups. Then
\begin{equation}
\label{eq-main-bis}
\Lambda_{\WH}(G \times H)
=
\Lambda_{\WH}(G) \Lambda_{\WH}(H).
\end{equation}
\end{thm}

\begin{proof}
We first prove the upper estimate. Suppose that both $\Lambda_{\WH}(G)$ and $\Lambda_{\WH}(H)$ are finite. Choose nets $(u_\alpha)$ in $\B_2(G) \cap \C_0(G)$ and $(v_\beta)$ in $\B_2(H) \cap \C_0(H)$ converging to $1_G$ and $1_H$, respectively, uniformly on compact subsets, and satisfying
\begin{equation}
\label{inter-21}
\norm{u_\alpha}_{\mathrm{B}_2(G)}
\leq
\Lambda_{\WH}(G)+\epsi
\quad \text{and} \quad
\norm{v_\beta}_{\B_2(H)}
\leq
\Lambda_{\WH}(H)+\epsi.
\end{equation}
The product net $(u_\alpha \ot v_\beta)_{(\alpha,\beta)}$ belongs to the space $\B_2(G \times H) \cap \C_0(G \times H)$ and converges to $1_{G \times H}$ uniformly on compact subsets of the group $G \times H$. We have
\begin{equation*}
\norm{u_\alpha \ot v_\beta}_{\B_2(G \times H)}
\ov{\eqref{eq-product-B2}}{\leq} \norm{u_\alpha}_{\B_2(G)} \norm{v_\beta}_{\B_2(H)}
\ov{\eqref{inter-21}}{\leq}
(\Lambda_{\WH}(G)+\epsi)
(\Lambda_{\WH}(H)+\epsi).
\end{equation*}
Letting $\epsi \to 0$, we obtain the inequality
\begin{equation}
\label{eq-upper-main-proof}
\Lambda_{\WH}(G \times H)
\leq
\Lambda_{\WH}(G) \Lambda_{\WH}(H).
\end{equation}
If one of the two constants is infinite, this inequality is automatic.

Now, we prove the reverse estimate. Assume first that $\Lambda_{\WH}(G) < \infty$ and $\Lambda_{\WH}(H) < \infty$. Let $0<\epsi<\min\{\Lambda_{\WH}(G),\Lambda_{\WH}(H)\}$. By Proposition \ref{prop-polar}, there exist functions $f \in \C_c(G)$ and $g \in \C_c(H)$ such that $p_G(f) \leq 1$, $p_H(g) \leq 1$,
\begin{equation}
\label{eq-near-polar}
\left|\int_G f(s) \d\mu_G(s)\right| 
> \Lambda_{\WH}(G)-\epsi,
\quad \text{and} \quad
\left|\int_H g(t) \d\mu_H(t)\right| 
> \Lambda_{\WH}(H)-\epsi.
\end{equation}
Proposition \ref{prop-p-tensor} gives
\begin{equation*}
p_{G\times H}(f \ot g)
\ov{\eqref{eq-p-tensor}}{=}  p_G(f)p_H(g)
\leq 1.
\end{equation*}
Applying Proposition \ref{prop-polar} to the group $G \times H$, and using Fubini's theorem, we obtain
\begin{align*}
\MoveEqLeft
\Lambda_{\WH}(G \times H)
\geq \left| \int_{G \times H} f(s)g(t) \d(\mu_G \ot \mu_H)(s,t) \right|\\
&=
\left|\int_G f(s) \d\mu_G(s)\right|
\left|\int_H g(t) \d\mu_H(t)\right|
\ov{\eqref{eq-near-polar}}{>}
(\Lambda_{\WH}(G)-\epsi)(\Lambda_{\WH}(H)-\epsi).
\end{align*}
Letting $\epsi \to 0$, we deduce that
\begin{equation}
\label{eq-lower-main-proof}
\Lambda_{\WH}(G \times H)
\geq
\Lambda_{\WH}(G) \Lambda_{\WH}(H).
\end{equation}

Finally, suppose that one of the two constants is infinite, say $\Lambda_{\WH}(G)=\infty$. The subgroup $G\times\{e_H\}$ is closed in $G\times H$ and is isomorphic to $G$. Thus
\begin{equation*}
\infty
=
\Lambda_{\WH}(G)
\ov{\eqref{eq-subgroup-WH}}{\leq} 
\Lambda_{\WH}(G \times H).
\end{equation*}
So $\Lambda_{\WH}(G \times H)=\infty$. Combining this observation with \eqref{eq-upper-main-proof} and \eqref{eq-lower-main-proof} proves \eqref{eq-main} in all cases.
\end{proof}

By induction, we immediately obtain the following extension.

\begin{cor}
\label{cor-finite-products}
Let $G_1,\ldots,G_n$ be locally compact groups. Then
\begin{equation*}
\Lambda_{\WH}(G_1 \times \cdots \times G_n)
=
\prod_{k=1}^n \Lambda_{\WH}(G_k).
\end{equation*}
\end{cor}

\section{Consequences and remarks}
\label{sec:consequences}

If $\Gamma$ is a discrete group, Knudby proved in \cite[Theorem B p.~3473]{Knu16} that $\Gamma$ has the weak Haagerup property if and only if its group von Neumann algebra $\VN(\Gamma)$ has the weak Haagerup property. More precisely, he proved that
\begin{equation}
\label{eq-group-vn}
\Lambda_{\WH}(\Gamma)
=\Lambda_{\WH}(\VN(\Gamma)),
\end{equation}
where the group von Neumann algebra $\VN(\Gamma)$ is equipped with its canonical normal faithful  tracial state.
\begin{cor}
\label{cor-vn-products}
Let $\Gamma_1$ and $\Gamma_2$ be discrete groups. Then
\begin{equation*}
\Lambda_{\WH}(\VN(\Gamma_1) \otvn \VN(\Gamma_2))
=\Lambda_{\WH}(\VN(\Gamma_1)) \Lambda_{\WH}(\VN(\Gamma_2)).
\end{equation*}
\end{cor}

\begin{proof}
The canonical trace preserving $*$-isomorphism $\VN(\Gamma_1 \times \Gamma_2) \to \VN(\Gamma_1) \otvn \VN(\Gamma_2)$, $\lambda_{(s,t)}\mapsto \lambda_s \ot \lambda_t$ combined with \eqref{eq-group-vn} and Theorem \ref{thm-main}, gives
\begin{align*}
\MoveEqLeft
\Lambda_{\WH}(\VN(\Gamma_1) \otvn \VN(\Gamma_2))
=\Lambda_{\WH}(\VN(\Gamma_1 \times \Gamma_2))
\ov{\eqref{eq-group-vn}}{=} \Lambda_{\WH}(\Gamma_1 \times \Gamma_2) \\
&\ov{\eqref{eq-main}}{=}
\Lambda_{\WH}(\Gamma_1) \Lambda_{\WH}(\Gamma_2)
\ov{\eqref{eq-group-vn}}{=} \Lambda_{\WH}(\VN(\Gamma_1)) \Lambda_{\WH}(\VN(\Gamma_2)).
\end{align*}
This proves the assertion.
\end{proof}


\paragraph{Competing interests} The author declares that he has no competing interests.

\paragraph{Data availability} No data sets were generated during this study.

{\footnotesize

\vspace{0.2cm}

\noindent C\'edric Arhancet\\
\noindent 6 rue Didier Daurat, 81000 Albi, France\\
URL: \href{http://sites.google.com/site/cedricarhancet}{https://sites.google.com/site/cedricarhancet}\\
cedric.arhancet@protonmail.com\\
ORCID: 0000-0002-5179-6972

}

\end{document}